\documentclass{proc-l}
\numberwithin{equation}{section}
\usepackage{hyperref}

\usepackage [latin1]{inputenc}

\usepackage{enumerate}

\numberwithin{table}{section}
\numberwithin{figure}{section}

\newtheorem{thm}{\bf Theorem}[section]

\newtheorem{cor}[thm]{\bf Corollary}

\newtheorem{prop}[thm]{Proposition}

\newtheorem{lem}[thm]{\bf Lemma}

\newtheorem{rem}[thm]{\bf Remark}

\newtheorem{exam}[thm]{\bf Example}

\allowdisplaybreaks

\newcommand{\normmm}[1]{{\left\vert\kern-0.25ex\left\vert\kern-0.25ex\left\vert #1
   \right\vert\kern-0.25ex\right\vert\kern-0.25ex\right\vert}}
\newcommand{\normm}[1]{{\vert\kern-0.25ex\vert\kern-0.25ex\vert #1
   \vert\kern-0.25ex\vert\kern-0.25ex\vert}}

\theoremstyle{definition}

\theoremstyle{remark}

\numberwithin{equation}{section}

\begin{document}

\title{Some nonlinear characterizations of reflexive Banach spaces}


\author{Yan Tang}
\address[Y. Tang]{School of Mathematics and Statistics,
Chongqing Technology and Bussiness University,
Chongqing 400067, China}
\address[Y. Tang]{College of Mathematics, Sichuan University, Chengdu 610064, China}
\curraddr{}
\email{tangyan@ctbu.edu.cn}
\thanks{The research is supported by the National Natural Science Foundation of China (No.11671278 and No.11971483) and the Science and Technology Research Project of Chongqing Municipal Education Commission(KJ1706154).}

\author{Shiqing Zhang}
\address[S. Q. Zhang]{College of Mathematics, Sichuan University, Chengdu 610064, China}
\curraddr{}
\email{zhangshiqing@msn.com}
\thanks{}

\author{Tiexin Guo$^*$}
\address[T. X. Guo]{School of Mathematics and Statistics, Central South University, Changsha 410083, China}
\curraddr{}
\email{tiexinguo@csu.edu.cn (Corresponding author)}
\thanks{$^*$ Corresponding author}


\date{}

\dedicatory{}


\begin{abstract}
It is well known that in the calculus of variations and in optimization there exist many formulations of the fundamental propositions on the attainment of the infima of sequentially weakly lower semicontinuous coercive functions on reflexive Banach spaces. By either some constructive skills or the regularization skill by inf--convolutions we show in this paper that all these formulations together with their important variants are equivalent to each other and equivalent to the reflexivity of the underlying space. Motivated by this research, we also give a characterization for a normed space to be finite dimensional: a normed space is finite dimensional iff every continuous real--valued function defined on each bounded closed subset of this space can obtain its minimum, namely the converse of the classical Weierstrass theorem also holds true.
\end{abstract}

\maketitle





\section{Introduction and the main results of this paper}
\label{intro}
\par
Throughout this paper, we always assume that normed or Banach spaces occurring in this paper are over the real number field $\mathbb{R}$ in order to simplify the writing, in fact the corresponding results are still true for normed or Banach spaces over the complex number field $\mathbb{C}$. Besides, $\mathbb{N}$ stands for the set of positive integers.
\par
For clarity, let us first recall some known terminologies as follows.
\par
Let $(X,\|\cdot\|)$ be a normed space, $A$ a nonempty subset of $X$ and $f: A\rightarrow (-\infty, +\infty]$ a function. $f$ is said to be:
\begin{enumerate}[(1)]
  \item proper if $f(x_0)<+\infty$ for some $x_0\in A$.
  \item bounded if $f(D)$ is bounded for each bounded subset $D$ of $A$ (at this time $f$ must be real--valued).
  \item lower semicontinuous (briefly, l.s.c) at $x_0\in A$ if $\underline{\lim}_n f(x_n)\geq f(x_0)$ for any sequence $\{x_n, n\in \mathbb{N}\}$ in $A$ such that $\lim_n\|x_n-x_0\|=0$. Furthermore, $f$ is l.s.c (on $A$) if it is l.s.c at any $x\in A$.
  \item weakly lower semicontinuous (briefly, w.l.s.c) at $x_0\in A$ if $\underline{\lim}_{\alpha}f(x_{\alpha})\geq f(x_0)$ for any net $\{x_{\alpha}, \alpha\in \Gamma\}$ in $A$ such that $\{x_{\alpha}, \alpha\in \Gamma\}$ weakly converges to $x_0$. Further, $f$ is w.l.s.c if it is w.l.s.c at any $x\in A$.
  \item sequentially weakly lower semicontinuous (briefly, s.w.l.s.c) at $x_0\in A$ if $\underline{\lim}_n f(x_n)$ $\geq f(x_0)$ for any sequence $\{x_n, n\in \mathbb{N}\}$ in $A$ such that $\{x_n, n\in \mathbb{N}\}$ weakly converges to $x_0$. Further, $f$ is s.w.l.s.c if it is s.w.l.s.c at any $x\in A$.
  \item coercive if $A$ is unbounded and $\lim_{x\in A, \|x\|\rightarrow +\infty} f(x)=+\infty$.
\end{enumerate}
\par
The boundedness of $f$ in the sense of (2) as above, which is in accordance with the usual boundedness of a bounded linear functional, is different from another kind of boundedness of $f$ on $A$, namely $\sup\{|f(x)|: x\in A\}<+\infty$. Although the latter kind of boundedness is not used in this paper, we still say that $f$ is bounded (from) below on $A$ if $-\infty<\inf\{f(x): x\in A\}<+\infty$ in accordance with the references cited in this paper.
\par
The classical Weierstrass theorem states that a continuous real--valued function defined on a nonempty bounded closed subset of a finite--dimensional normed space attains its minimum and maximum, in 1870 Weierstrass first gave a counterexample showing that his theorem does not hold in an infinite--dimensional Banach space, which forces people to take seriously and subsequently systematically develop the direct method of calculus of variations, see \cite{Maw,Zhang} for the excellent historical comments. Today, it is well known that in the calculus of variations and in optimization there exist many (at least seven) propositions on the attainment of the infima of sequentially weakly lower semicontinuous coercive functions on reflexive Banach spaces. The purpose of this paper is to prove that all these propositions are equivalent to each other and equivalent to the reflexivity of the underlying space.
\par
To conveniently describe the idea of this paper, we restate in detail five of the well known seven propositions (the other two are also mentioned in passing) but keep their original meaning.
\par
For the reader's convenience, let us first recall that, according to Mazur's lemma, $A$ is closed $\Leftrightarrow$ $A$ is weakly closed $\Leftrightarrow$ $A$ is sequentially weakly closed for any nonempty convex subset $A$ of a normed space $E$, and hence $f$ is l.s.c $\Leftrightarrow$ $f$ is w.l.s.c $\Leftrightarrow$ $f$ is s.w.l.s.c for a convex function $f: A\rightarrow (-\infty,+\infty]$ when $A$ is a nonempty closed convex subset of $E$.
\par
Proposition \ref{proposition1.1} below initiated the abstract approach to problems in the calculus of variations, which is essentially due to Mazur and Schauder \cite{Mau,MS} who originally only considered a lower semicontinuous coercive real--valued lower bounded convex function defined on an unbounded closed convex subset of a reflexive Banach space and whose proof was by the KKM mapping principle. In fact, the condition ``lower bounded, namely bounded below'' is not necessarily assumed in advance since the condition is automatically satisfied, see \cite{Bre,Eke} for the case of a reflexive space and also Lemma \ref{lemma2.2} of this paper for the case of a general normed space.

\begin{prop}[{\cite[Corollary 3.23]{Bre}}; {\cite[Proposition 1.2]{Eke}}]\label{proposition1.1}
Let $E$ be a reflexive Banach space, then the following two statements hold:
\begin{enumerate}[(1)]
  \item For each nonempty bounded closed convex subset $A$ of $E$ and each proper l.s.c convex function $f: A\rightarrow (-\infty,+\infty]$, there exists $x_0\in A$ such that $f(x_0)=\inf_{x\in A}f(x)$.
  \item For each unbounded closed convex subset $A$ of $E$ and each proper l.s.c coercive convex function $f: A\rightarrow (-\infty,+\infty]$, there exists $x_0\in A$ such that $f(x_0)=\inf_{x\in A}f(x)$.
\end{enumerate}
\end{prop}

\par
Since Eberlin and Shmulyan established the characterization of a reflexive Banach space in 1940s, namely a Banach space is reflexive iff each bounded sequence of $E$ admits a weakly convergent subsequence (see \cite{Yas}), many famous mathematicians began considering the problem of the attainment of the infima of sequentially weakly lower semicontinuous functions defined on sequentially weakly closed subsets of a reflexive Banach space, so that Proposition \ref{proposition1.2} below was obtained as a generalization of (2) of Proposition \ref{proposition1.1}.

\begin{prop}[see {\cite{Chang}} or {\cite[Theorem 4.3.8]{Zhang}}]\label{proposition1.2}
Let $E$ be a reflexive Banach space. Then for each unbounded sequentially weakly closed subset $A$ of $E$ and each proper s.w.l.s.c coercive function $f: A\rightarrow (-\infty, +\infty]$, there exists $x_0\in A$ such that $f(x_0)=\inf_{x\in A}f(x).$
\end{prop}

\par
A slightly less general formulation of Proposition \ref{proposition1.2} was also mentioned, e.g. in \cite{Guo} $f$ is restricted to be real--valued. It is obvious that Propositions \ref{proposition1.3} and \ref{proposition1.4} below are both a special case of Proposition \ref{proposition1.2}. Besides, Proposition \ref{proposition1.3} is a slight generalization of Theorem 1.1 of \cite[Chapter 3]{Dac}, where the original Theorem 1.1 only considered the function $\varphi: [0,+\infty)\rightarrow [0,+\infty)$ defined by $\varphi(t)=\alpha t$ for any $t\geq0$ and for some positive number $\alpha$, but in the process of applying it to some concrete examples, the function $\varphi: [0,+\infty) \rightarrow [0,+\infty)$ defined by $\varphi(t)=\alpha t^p$ for any $t\geq 0$ and for some positive numbers $\alpha$ and $p$ with $1<p<+\infty$, was also employed.

\begin{prop}[{\cite[Chapter 3, Theorem 1.1]{Dac}}]\label{proposition1.3}
Let $E$ be a reflexive Banach space. Then for each proper s.w.l.s.c function $f: E \rightarrow (-\infty,+\infty]$ satisfying $f(x)\geq \varphi(\|x\|)+ \beta$ for any $x\in E$ and for some real number $\beta$ and some nondecreasing function $\varphi: [0,+\infty) \rightarrow [0,+\infty )$ with $\lim_{t\rightarrow +\infty} \varphi(t)=+\infty$, there exists $x_0\in E$ such that $f(x_0)=\inf_{x\in E} f(x)$.
\end{prop}

\begin{prop}[{\cite[Theorem 6.1.1]{Ber}}]\label{proposition1.4}
Let $E$ be a reflexive Banach space. Then for each unbounded sequentially weakly closed subset $A$ of $E$ and each s.w.l.s.c bounded coercive function $f: A \rightarrow (-\infty,+\infty)$, there exists $x_0\in A$ such that $f(x_0)=\inf_{x\in A}f(x)$.
\end{prop}

\par
Let $A$ be a nonempty subset of a normed space $E$ and $f: A \rightarrow (-\infty,+\infty]$ be an extended real--valued function, let us recall that a sequence $\{x_n,n\in \mathbb{N}\}$ in $A$ is a minimizing sequence of $f$ if $\{f(x_n), n\in \mathbb{N}\}$ converges to $\inf_{x\in A}f(x)$ in a nonincreasing manner. It is well known that Propositions \ref{proposition1.1}--\ref{proposition1.4} together with their less general variants mentioned above all can be proved by means of the Eberlin--Shmulyan theorem, and the essence of their proofs comes down to the existence of a bounded minimizing sequence of $f$ in these propositions, this observation leads Mawhin and Willem directly to Proposition \ref{proposition1.5} below in a more refined and general form.

\begin{prop}[{\cite[Theorem 1.1]{Maw}}]\label{proposition1.5}
Let $E$ be a reflexive Banach space. Then for each s.w.l.s.c proper function $f: E \rightarrow (-\infty,+\infty]$ with a bounded minimizing sequence, there exists $x_0\in E$ such that $f(x_0)=\inf_{x\in E}f(x)$.
\end{prop}

\par
For a nonempty subset $K$ of a normed space $E$, its indicator function $I_K: E \rightarrow [0,+\infty]$ is defined by

$$ I_K(x)=\left\{
\begin{aligned}
&0, &\text{when}~ x\in K;  \\
&+\infty,  & \text{when}~x\in E\setminus K.
\end{aligned}
\right.
$$
It is well known that $K$ is convex $\Leftrightarrow$ $I_K$ is convex, $K$ is closed $\Leftrightarrow$ $I_K$ is l.s.c, $K$ is weakly closed $\Leftrightarrow$ $I_K$ is w.l.s.c, and $K$ is sequentially weakly closed $\Leftrightarrow$ $I_K$ is s.w.l.s.c.
\par
The classical James theorem \cite{Jam} states that a Banach space $E$ is reflexive iff there exists $x\in B$ for each $g\in E^*$ such that $g(x)=\|g\|$, where $B$ stands for the closed unit ball of $E$, namely $B=\{x\in E: \|x\|\leq 1\}$. According to the James theorem, it is very easy to observe that the conclusions of Propositions \ref{proposition1.1}, \ref{proposition1.2} and \ref{proposition1.5} have implied the reflexivity of the underlying space $E$, respectively, namely the reflexivity of $E$ is both sufficient and necessary. Let us only take Proposition \ref{proposition1.1} for example: for any $g\in E^*$, taking $A:=B$ (the closed unit ball of $E$) and $f:=-g|_B$ (where $g|_B$ stands for the restriction of $g$ to $B$) in (1) of Proposition \ref{proposition1.1} yields some $x_0\in B$ such that $g(x_0)=\sup_{x\in B} g(x)=\|g\|$, and hence $E$ must be reflexive by the James theorem; similarly, for any $g\in E^*$, taking $A:=E$ and $f:=-g+I_B$ in (2) of Proposition \ref{proposition1.1} yields some $x_0\in B$ such that $g(x_0)=\|g\|$, then $E$ is again reflexive.

\par
But it is not very easy to solve the problem of whether the conclusions of Propositions \ref{proposition1.3} and \ref{proposition1.4} also imply the reflexivity of the underlying space $E$, respectively, since the functions $f$ in the two propositions are both a special class of coercive functions and their links with such functions as $-g+I_B$ with $g\in E^*$ are no longer obvious. Dacorogna earlier pointed out in \cite[Remark, p.48]{Dac} that the hypothesis on the reflexivity of $E$ cannot be dropped in general. Example \ref{example1.6} below, which is only a slight modification to the original example given in \cite[p.48]{Dac}, not only better explains Dacorogna's viewpoint but also motivates some constructive skills.

\begin{exam}\label{example1.6}
Let $E=C[0,1]$, namely the Banach space of continuous real--valued functions on $[0,1]$ endowed with the norm $\|\cdot\|$ defined by $\|u\|=\max_{t\in [0,1]} |u(t)|$ for any $u\in C[0,1]$, it is known that $C[0,1]$ is not reflexive. Further, let $M=\{u\in C[0,1]: \int_0^{\frac{1}{2}} u(t) dt - \int^1_{\frac{1}{2}} u(t) dt= 1\}$ and define $f: C[0,1] \rightarrow [0,+\infty]$ as follows:
$$ f(u)=\left\{
\begin{aligned}
&\|u\|, &\text{if}~ u\in M;  \\
&+\infty,  & \text{if}~u\in C[0,1]\setminus M.
\end{aligned}
\right.
$$
Then $f$ is a proper l.s.c convex (and hence also s.w.l.s.c) function since $M$ is an unbounded closed convex subset of $C[0,1]$ as a closed hyperplane, and further satisfies $f(u) \geq \varphi(\|u\|)+ \beta$ for any $u\in C[0,1]$ and for the function $\varphi: [0,+\infty) \rightarrow [0,+\infty)$ with $\varphi(t)=t$ for any $t\geq 0$, and $\beta=0$. It is easy to see that $\inf_{u\in M} f(u)=\inf_{u\in C[0,1]} f(u)$. Further, since $f(u)=\|u\|\geq \int_0^1 |u(t)| dt \geq \int_0^{\frac{1}{2}} u(t) dt- \int^1_{\frac{1}{2}} u(t) dt =1$ for any $u\in M$, then $\inf_{u\in M} f(u) \geq 1$. To work out $\inf_{u\in M} f(u)=1$, define a sequence $\{\tilde{u}_n, n\in \mathbb{N}\}$ in $C[0,1]$ as follows:
$$ \tilde{u}_n(t)=\left\{
\begin{aligned}
&1, &\text{if}~ t\in [0, \frac{1}{2}- \frac{1}{n+1}];  \\
&-(n+1)t+ \frac{n+1}{2}, & \text{if}~t\in [\frac{1}{2}- \frac{1}{n+1}, \frac{1}{2}+ \frac{1}{n+1}]; \\
&-1,  & \text{if}~t\in [\frac{1}{2}+ \frac{1}{n+1}, 1].
\end{aligned}
\right.
$$
Then it is easy to check that $\int_0^{\frac{1}{2}} \tilde{u}_n(t) dt- \int^1_{\frac{1}{2}} \tilde{u}_n(t) dt= 1-\frac{1}{n+1}$, then $u_n: = \tilde{u}_n/ (1- \frac{1}{n+1}) \in M$ and it is also clear that $\|u_n\|= \frac{n+1}{n}$, so $\inf_{u\in M} f(u) \leq \inf_{n\in \mathbb{N}} \frac{n+1}{n} =1$, which shows $\inf_{u\in M} f(u)=1$. But, as shown in \cite[p.49]{Dac}, there exists no $u\in M$ such that $f(u)=1$.
\end{exam}

\par
Existence of Example \ref{example1.6} is no surprise, this paper will, in fact, show that the conclusions of Propositions \ref{proposition1.3} and \ref{proposition1.4} both imply the reflexivity of the underlying space $E$. Precisely, by either the regularization skill by inf--convolutions from convex analysis or some constructive skills we establish the relations among Proposition \ref{proposition1.3}, Proposition \ref{proposition1.4} and the James theorem so that our target can be achieved. Let $E$ be a normed space and $f: E\rightarrow (-\infty, +\infty]$ a function bounded below, it is well known that the Pasch--Hausdorff envelope $f_k$ of $f$ is a $k$--Lipschitz function for any positive number $k$ (see Section \ref{section2} for the related details and references), which is just the so--called regularization skill. Besides the Lipschitz property of $f_k$ is used in this paper, Proposition \ref{proposition2.5} of this paper, which is obtained under a considerable motivation from Hiriart--Urruty's work \cite{Hir1} and Bougeard, Penot and Pommellet's work \cite{BPP}, really plays a crucial role in connecting the inf--convolution skill to the work of this paper since $f_k$ possesses the same infimum and global minimizers as $f$.

\par
When we have known that the conclusions in Propositions \ref{proposition1.1}--\ref{proposition1.5} characterize the reflexivity of the underlying space, we naturally want to give a similar characterization of a finite dimensional normed space, namely Theorem \ref{theorem1.9} of this paper, in particular Theorem \ref{theorem1.9} shows that the converse of the classical Weierstrass theorem is also true.

\par
The main results of this paper can now stated as follows: Theorem \ref{theorem1.7} below is devoted to the case of convex functions, Theorem \ref{theorem1.8} to the case of general nonlinear functions, Theorem \ref{theorem1.9} to a similar characterization for a finite--dimensional normed space, and Corollary \ref{corollary1.10} is an interesting consequence of comparing Theorem \ref{theorem1.8} and Theorem \ref{theorem1.9}.

\begin{thm}\label{theorem1.7}
Let $E$ be a Banach space. Then the following statements are equivalent to each other:
\begin{enumerate}[(1)]
  \item $E$ is reflexive.
  \item For each proper l.s.c convex function $f: E \rightarrow (-\infty,+\infty]$ such that $f$ has a bounded minimizing sequence, there exists $x_0\in E$ such that $f(x_0)=\inf_{x\in E} f(x)$.
  \item For each bounded nonempty closed convex subset $A$ of $E$ and each proper l.s.c convex function $f: A \rightarrow (-\infty,+\infty]$, there exists $x_0\in A$ such that $f(x_0)=\inf_{x\in A} f(x)$.
  \item For each unbounded closed convex subset $A$ of $E$ and each coercive l.s.c convex function $f: A \rightarrow (-\infty,+\infty)$, there exists $x_0\in A$ such that $f(x_0)=\inf_{x\in A} f(x)$.
  \item For each unbounded closed convex subset $A$ of $E$ and each bounded coercive l.s.c convex function $f: A \rightarrow (-\infty, +\infty)$, there exists $x_0\in A$ such that $f(x_0)= \inf_{x\in A} f(x)$.
  \item For each coercive convex Lipschitz function $f: E \rightarrow (-\infty, +\infty)$, there exists $x_0\in E$ such that $f(x_0)=\inf_{x\in E} f(x)$.
  \item For each unbounded closed convex subset $A$ of $E$ and each proper l.s.c coercive convex function $f: A \rightarrow (-\infty, +\infty]$, there exists $x_0\in A$ such that $f(x_0)=\inf_{x\in A} f(x)$.
  \item For each proper l.s.c convex function $f: E \rightarrow (-\infty,+\infty]$ satisfying $f(x) \geq \varphi(\|x\|) + \beta$ for any $x\in E$ and for some real number $\beta$ and some nondecreasing function $\varphi: [0,+\infty) \rightarrow [0,+\infty)$ with $\lim_{t\rightarrow +\infty} \varphi(t) =+\infty$, there exists $x_0\in E$ such that $f(x_0)=\inf_{x\in E} f(x)$.
  \item For each bounded linear functional $g$ on $E$ there exists $x_0\in E$ with $\|x_0\| \leq 1$ such that $g(x_0)=\|g\|$.
\end{enumerate}
\end{thm}

\begin{thm}\label{theorem1.8}
Let $E$ be a Banach space. Then the following statements are equivalent to each other:
\begin{enumerate}[(1)]
  \item $E$ is reflexive.
  \item For each proper s.w.l.s.c function $f: E \rightarrow (-\infty, +\infty]$ such that $f$ has a bounded minimizing sequence, there exists $x_0\in E$ such that $f(x_0)=\inf_{x\in E} f(x)$.
  \item For each bounded nonempty sequentially weakly closed subset $A$ of $E$ and each proper s.w.l.s.c function $f: A \rightarrow (-\infty, +\infty]$, there exists $x_0\in A$ such that $f(x_0)=\inf_{x\in A} f(x)$.
  \item For each unbounded sequentially weakly closed subset $A$ of $E$ and each proper s.w.l.s.c coercive function $f: A \rightarrow (-\infty,+\infty]$, there exists $x_0\in A$ such that $f(x_0)=\inf_{x\in A} f(x)$.
  \item For each unbounded sequentially weakly closed subset $A$ of $E$ and each s.w.l.s.c coercive function $f: A \rightarrow (-\infty,+\infty)$, there exists $x_0\in A$ such that $f(x_0)=\inf_{x\in A}f(x)$.
  \item For each proper s.w.l.s.c function $f: E \rightarrow (-\infty,+\infty]$ satisfying $f(x)\geq \varphi(\|x\|)+ \beta$ for any $x\in E$ and for some real number $\beta$ and some nondecresing function $\varphi: [0,+\infty) \rightarrow [0,+\infty)$ with $\lim_{t \rightarrow +\infty} \varphi (t)=+\infty$, there exists $x_0\in E$ such that $f(x_0)=\inf_{x\in E} f(x)$.
  \item For each unbounded sequentially weakly closed subset $A$ of $E$ and each s.w.l.s.c bounded coercive function $f: A \rightarrow (-\infty,+\infty)$, there exists $x_0\in A$ such that $f(x_0)=\inf_{x\in A} f(x)$.
  \item For each s.w.l.s.c coercive Lipschitz function $f: E \rightarrow (-\infty, +\infty)$ there exists $x_0\in E$ such that $f(x_0)=\inf_{x\in E} f(x)$.
  \item For each bounded linear functional $g$ on $E$ there exists $x_0\in E$ with $\|x_0\|\leq 1$ such that $g(x_0)=\|g\|$.
\end{enumerate}
\end{thm}

\par
(1) $\Rightarrow$ (4) of Theorem \ref{theorem1.9} below, whose proof is very easy, was already stated in \cite{Chang}, (1) $\Rightarrow$ (9) is just the classical Weierstrass theorem. The most interesting part of the proof of Theorem \ref{theorem1.9} is (9) $\Rightarrow$ (1) or (10) $\Rightarrow$ (1).

\begin{thm}\label{theorem1.9}
Let $E$ be a normed space. Then the following statements are equivalent to each other:
\begin{enumerate}[(1)]
  \item $E$ is finite--dimensional.
  \item For each proper l.s.c function $f: E \rightarrow (-\infty,+\infty]$ such that $f$ has a bounded minimizing sequence, there exists $x_0\in E$ such that $f(x_0)=\inf_{x\in E}f(x)$.
  \item For each bounded nonempty closed subset $A$ of $E$ and each proper l.s.c function $f: A \rightarrow (-\infty,+\infty]$, there exists $x_0\in A$ such that $f(x_0)=\inf_{x\in A}f(x)$.
  \item For each unbounded closed subset $A$ of $E$ and each proper coercive l.s.c function $f: A \rightarrow (-\infty,+\infty]$, there exists $x_0\in A$ such that $f(x_0)=\inf_{x\in A} f(x)$.
  \item For each unbounded closed subset $A$ of $E$ and each coercive l.s.c function $f: A \rightarrow (-\infty,+\infty)$, there exists $x_0\in A$ such that $f(x_0)=\inf_{x\in A} f(x)$.
  \item For each proper l.s.c function $f: E \rightarrow (-\infty,+\infty]$ satisfying $f(x)\geq \varphi(\|x\|)+ \beta$ for any $x\in E$ and for some real number $\beta$ and some nondecreasing function $\varphi: [0,+\infty) \rightarrow [0,+\infty)$ with $\lim_{t\rightarrow +\infty} \varphi(t)=+\infty$, there exists $x_0\in E$ such that $f(x_0)=\inf_{x\in E} f(x)$.
  \item For each unbounded closed subset $A$ of $E$ and each bounded coercive l.s.c function $f: A \rightarrow (-\infty,+\infty)$, there exists $x_0\in A$ such that $f(x_0)=\inf_{x\in A}f(x)$.
  \item For each coercive Lipschitz function $f: E \rightarrow (-\infty,+\infty)$, there exists $x_0\in E$ such that $f(x_0)=\inf_{x\in E} f(x)$.
  \item For each bounded closed nonempty subset $A$ of $E$ and each continuous real--valued function $f$ on $A$, there exists $x_0\in A$ such that $f(x_0)=\inf_{x\in A} f(x)$.
  \item For each nonempty bounded closed subset $A$ of $E$ and each continuous real--valued function $f$ defined on $A$ such that $f$ is bounded below, there exists $x_0\in A$ such that $f(x_0)=\inf_{x\in A} f(x)$.
\end{enumerate}
\end{thm}

\begin{cor}\label{corollary1.10}
For each infinite--dimensional reflexive Banach space $E$, there exists at least one coercive Lipschtiz function $f$ on $E$ such that $f$ is not s.w.l.s.c.
\end{cor}

\par
For any Banach space $E$, the connection between s.w.l.s.c functions and s.w.l.s.c Lipschtiz functions on $E$ by inf--convolutions will be further discussed in Section \ref{section2} of this paper.

\par
The remainder of this paper is organized as follows. Section \ref{section2} is devoted to some preliminaries used in the proofs of the main results: where we prove that a proper l.s.c coercive convex function on any normed space must be bounded below; further, we also discuss the properties of inf--convolutions of l.s.c and s.w.l.s.c functions. Section \ref{section3} is devoted to the proofs of the main results and ends with some important remarks.

\section{Preliminaries on convex functions and inf--convolutions}\label{section2}
Lemma \ref{lemma2.1} below is known, see \cite[Proposition 1.10]{Bre} and \cite{Eke,Phe}, which shows the nice property of convex functions.

\begin{lem}[{\cite{Bre,Eke,Phe}}]\label{lemma2.1}
Let $(X,\|\cdot\|)$ be a normed space and $f: X \rightarrow (-\infty,+\infty]$ a proper l.s.c convex function, then there exist $x^*\in X^*$ and $\alpha\in \mathbb{R}$ such that $f(x)\geq x^*(x)+ \alpha$ for each $x\in X$.
\end{lem}

\par
Lemma \ref{lemma2.2} below is known for a reflexive Banach space and pointed out in the process of the proof of Proposition \ref{proposition1.1}. In fact, it always holds for any normed space, which will both considerably simplify the proofs of the well known results of l.s.c convex functions and play a crucial role in the proofs of the main results since it helps produce convex Lipschitz coercive functions.

\begin{lem}\label{lemma2.2}
Let $(X,\|\cdot\|)$ be a normed space, $A$ a nonempty closed convex subset of $X$ and $f: A \rightarrow (-\infty,+\infty]$ a proper l.s.c convex function, then $f$ is bounded below if either of the following two conditions is satisfied:
\begin{enumerate}[(1)]
  \item $A$ is bounded.
  \item $A$ is unbounded and $f$ is coercive.
\end{enumerate}
\end{lem}

\begin{proof}
Under the condition (2), there exists $x_0\in A$ such that $f(x_0)<+\infty$ since $f$ is proper, then $A_1=\{x\in A: f(x)\leq f(x_0)\}$ is a nonempty bounded closed convex subset of $A$ by the coercivity of $f$ and it is also obvious that $\inf_{x\in A}f(x)=\inf_{x\in A_1} f(x)$. By noticing that $f|_{A_1}: A_1 \rightarrow (-\infty,+\infty)$ is still l.s.c (of course, also proper) we only need to prove the lemma under the condition (1) as follows.
\par
First, we extend $f$ to $\tilde{f}$ as follows:
$$ \tilde{f}(x)=\left\{
\begin{aligned}
&f(x), &\text{if}~ x\in A;  \\
&+\infty, & \text{if}~x\in X\setminus A.
\end{aligned}
\right.
$$
Then $\tilde{f}$ is still a proper l.s.c convex function and $\inf_{x\in A} f(x)= \inf_{x\in X} \tilde{f}(x)$. By Lemma \ref{lemma2.1} there exists $x^*\in X^*$ and $\alpha \in \mathbb{R}$ such that $\tilde{f}(x)\geq x^*(x)+ \alpha$ for any $x\in X$, and hence $f(x)\geq x^*(x)+ \alpha \geq -\|x^*\|\|x\|+ \alpha$ for any $x\in A$. Let $M=\sup\{\|x\|: x\in A\}$, then $f(x)\geq -\|x^*\|M+ \alpha$ for any $x\in A$, so $\inf_{x\in A}f(x)\in (-\infty,+\infty)$.
\end{proof}

\begin{rem}\label{remark2.3}
It is very easy to construct an example to show that Lemma \ref{lemma2.2} is not true for nonconvex functions. The key ingredient in the proof of Lemma \ref{lemma2.2} is the use of Lemma \ref{lemma2.1}, namely there exists an affine function $x^*+ \alpha$ which minorizes $f$, since the fact still holds for an arbitrary proper (not necessarily l.s.c) convex function $f: \mathbb{R}^n \rightarrow (-\infty,+\infty]$, see \cite[Proposition IV.1.2.1]{HL}, Lemma \ref{lemma2.1} also holds in finite dimensional normed spaces when the lower semicontinuity of $f$ is removed.
\end{rem}

\par
The operation of inf--convolutions was originally introduced and studied by Hausdorff, Fenchel, Moreau and Rockafellar, see \cite{Phe,Roc,RW} for the related historical comments. Let $(E,\|\cdot\|)$ be a normed space and $f$ and $g$ be two extended real--valued functions from $E$ to $(-\infty,+\infty]$, the inf--convolution $f\square g$ of $f$ and $g$ is an extended real--valued function from $E$ to $[-\infty,+\infty]$ defined by
$$(f \square g)(x)=\inf\{f(y)+ g(x-y): y\in E\}~ \text{for any}~ x\in E.$$

\par
For any positive number $k$, the Pasch--Hausdorff envelope of $f$ is defined as follows by
$$f_k(x)=(f \square k\|\cdot\|)(x)=\inf\{f(y)+ k\|y-x\|: y\in E\}~\text{for any}~x\in E,$$
and the Yosida--Moreau envelope of $f$ is defined by
$$f_{\langle k \rangle}(x)= (f\square \frac{k}{2}\|\cdot\|^2)(x)= \inf\{f(y)+ \frac{1}{2}k\|y-x\|^2: y\in E\}~\text{for any}~x\in E.$$

\par
When $k$ runs over $\mathbb{N}$, one can obtain the sequence $\{f_n: n\in \mathbb{N}\}$ of inf--convolutions of $f$ with respect to $n\|\cdot\|$. This paper only uses the sequence $\{f_n, n\in \mathbb{N}\}$ and for the sake of convenience we briefly call it the sequence of inf--convolutions of (or by) $f$. The sequence $\{f_n, n\in \mathbb{N}\}$ of inf--convolutions of $f$ provides a classical way of approximating $f$, which has been successfully used in the study of the subdifferentials, approximate subdifferentials, the Clarke subdifferentials, differentiability and extension of Lipschitz functions, see \cite{FP,Hir2,WC,Thi,Hir1}.

\par
Lemma \ref{lemma2.4} below surveys the basic properties of inf--convolutions, most of which are known but with one addition.

\begin{lem}\label{lemma2.4}
Let $(E,\|\cdot\|)$ be a normed space and $f: E \rightarrow (-\infty,+\infty]$ be a proper extended real--valued function. Then the following statements hold:
\begin{enumerate}[(1)]
  \item For each $x\in E$, $\{f_n(x), n\in \mathbb{N}\}$ is a nondecreasing sequence and $f_n(x)\leq f(x)$ for any $n\in \mathbb{N}$ (see \cite{Bre,Phe}).
  \item If $f$ is convex, then $f_k$ is convex for each positive number $k$ (see \cite{Bre,Phe}).
  \item If $f$ is convex or bounded below, then $\{f_n(x), n\in \mathbb{N}\}$ converges to $f(x)$ for each $x\in E$ iff $f$ is l.s.c (see \cite{Bre,Phe,WC}).
  \item If $f$ is bounded below, then $f_k$ is Lipschitz with Lipschitz constant $k$ for each positive number $k$ (see \cite{FP,Phe,WC}).
  \item If $f$ is convex and l.s.c, then $f_n$ is Lipschitz with Lipschitz constant $n$ for sufficiently large $n$ (see \cite{Bre,FP,Phe}).
  \item If $f$ is bounded below and there exists some nonempty bounded subset $H$ of $E$ such that $f(x)=+\infty$ for any $x\in E\setminus H$, then $f_k$ is coercive for each positive number $k$.
\end{enumerate}
\end{lem}

\begin{proof}
We only need to prove (6) since the others are known. Let $m=\inf_{x\in E} f(x)$ and $d(x,H)=\inf\{d(x,h): h\in H\}$, then
\begin{equation*}
 \aligned
f_k(x)& =\inf\{f(y)+k\|y-x\|: y\in E\} \\
&=\inf\{f(y)+k\|y-x\|: y\in H\} \\
&\geq m+kd(x,H).
 \endaligned
\end{equation*}
Then $\lim_{\|x\| \rightarrow +\infty} f_k(x)=+\infty$ since $H$ is bounded.
\end{proof}

\par
The Yosida--Moreau envelope $f_{\langle k\rangle}$ of a proper l.s.c function $f$ has been extensively studied since $f_{\langle k\rangle}$ possesses the nice smoothness and the same infimum and global minimizers as $f$, and also since $f_{\langle k\rangle}$ is closely related to the viscosity solutions of Hamilton--Jocobi equations, see \cite{Luc,HL,RW,BW,PB,BT,BPP,LL}. The Pasch--Hausdorff envelope $f_k$ of $f$ is often Lipschitz, which is what is needed in this paper. Besides, Hiriart--Urruty \cite{Hir1} earlier proved that $f_k$ could also retain the infimum and global minimizers of $f$ in the special case that $f$ is Lipschitz with Lipschitz constant $k$ and $f$ is defined on a closed subset of $E$ (see Corollary \ref{corollary2.7} below for his result). Proposition \ref{proposition2.5} below shows that $f_k$ still retains the infimum and global minimizers of $f$ for a general l.s.c function $f$ from $E$ to $(-\infty,+\infty]$. Proposition \ref{proposition2.5} seems known at least in the field of convex analysis (for example, the authors of \cite{BPP} also inexplicitly mentioned this result) but has not been explicitly stated and given a concrete proof, here we would like to give its proof since we want to emphasize that this result can not be generalized to nonconvex functions with value $-\infty$ at some point (see Remark \ref{remark2.6} below).

\begin{prop}\label{proposition2.5}
Let $(E,\|\cdot\|)$ be a normed space, $k$ a positive number, and $f: E \rightarrow (-\infty,+\infty]$ an extended real--valued function. Then the following statements hold:
\begin{enumerate}[(1)]
  \item $\inf_{x\in E}f(x)=\inf_{x\in E} f_k(x)$.
  \item Any global minimizer for $f$ is a global minimizer for $f_k$, and if $f$ is l.s.c any global minimizer for $f_k$ is also a global minimizer for $f$.
\end{enumerate}
\end{prop}

\begin{proof}
(1) is clear by noticing $f_k\leq f$ and $\inf_{x\in E} f_k(x) \geq \inf_{x\in E}f(x)$.
\par
(2). If $x_0$ is a global minimizer of $f$, then $f_k(x_0)\leq f(x_0)= \inf_{x\in E} f(x)= \inf_{x\in E} f_k(x)$, so $x_0$ is a global minimizer of $f_k$. Conversely, if $f$ is l.s.c and $x_0$ is a global minimizer of $f_k$, then in the case that $f\equiv +\infty$, also $f_k\equiv +\infty$, and hence $x_0$ must be a global minimizer for $f$. When $f$ is proper, $\inf_{x\in E} f(x)=\inf_{x\in E} f_k(x) =f_k(x_0)\in \mathbb{R}$, which means that $f$ is bounded below. Denote $m=\inf_{x\in E} f(x)$, by the definition of $f_k$ there must exist a sequence $\{x_n, n\in \mathbb{N}\}$ such that $\{f(x_n)+ k\|x_n-x_0\|, n\in \mathbb{N}\}$ converges to $m$ in a nonincreasing manner, which forces $\{f(x_n), n\in \mathbb{N}\}$ to tend to $m$ and $\{x_n, n\in \mathbb{N}\}$ to converge to $x_0$, so that $m\leq f(x_0)\leq \underline{\lim}_nf(x_n)= \lim_n f(x_n)= m$, namely $x_0$ is a global minimizer for $f$.
\end{proof}

\begin{rem}\label{remark2.6}
When $f$ is convex and takes the value $-\infty$ at some point, then once $f$ is l.s.c $f$ must be identically equal to $-\infty$ , at which time $f_k\equiv -\infty$, $f$ and $f_k$ possess, of course, the same global minimizers, namely (2) of Proposition \ref{proposition2.5} still holds for convex functions which can take the value $-\infty$. But for nonconvex functions which can take the value $-\infty$, (2) of Proposition \ref{proposition2.5} not necessarily holds, such a counterexample is given as follows:
$$ f(x)=\left\{
\begin{aligned}
&\ln x, &\text{if}~ x\in (0,+\infty);  \\
&-\infty, & \text{if}~x=0; \\
&+\infty, & \text{if}~x\in (-\infty,0).
\end{aligned}
\right.
$$
Then $f: (-\infty,+\infty) \rightarrow [-\infty,+\infty]$ is a l.s.c nonconvex function on $\mathbb{R}$. It is obvious that $f_k\equiv -\infty$ for any given positive number $k$, then any $x_0\neq 0$ is a global minimizer for $f_k$ but not a global minimizer for $f$.
\end{rem}

\begin{cor}[{\cite[Theorem 4(a)]{Hir1}}]\label{corollary2.7}
Let $(E,\|\cdot\|)$ be a normed space, $S$ a nonempty closed subset of $E$, and $f: S\rightarrow R$ a Lipschitz function on $S$ with Lipschitz constant $k>0$. Denote by $\tilde{f}_k$ the Pasch--Hausdorff envelope of $\tilde{f}$, then $x_0\in E$ satisfies $\tilde{f}_k(x_0)= \inf_{x\in E} \tilde{f}_k(x)$ iff $x_0\in S$ and $f(x_0)= \inf_{x\in S} f(x)$, where $\tilde{f}: E \rightarrow (-\infty, +\infty]$ is given by:
$$ \tilde{f}(x)=\left\{
\begin{aligned}
&f(x), &\text{if}~ x\in S;  \\
&+\infty, & \text{if}~x\in E\setminus S.
\end{aligned}
\right.
$$
\end{cor}

\begin{proof}
It is clear that $f$ is continuous on $S$, then $\tilde{f}$ must be l.s.c on $E$ since $S$ is closed. By observing that $\inf_{x\in S} f(x)= \inf_{x\in E} \tilde{f}(x)$, the proof immediately follows from Proposition \ref{proposition2.5}.
\end{proof}

\par
It immediately follows from Lemma \ref{lemma2.4} that $f_k$ is Lipschitz and hence preserves the lower semicontinuity of $f$ if $f$ is l.s.c and bounded below. Since the notion of a sequentially weakly lower semicontinuity is stronger than that of a lower semicontinuity, a natural problem is whether $f_k$ can also retain the sequentially weakly lower semicontinuity of $f$ if $f$ is s.w.l.s.c and bounded below. Currently, we only have a partial answer to the problem although we think that the final answer is ``No''. Theorem \ref{theorem2.8} below shows that the answer is Yes when the underlying space is reflexive, whereas the following Remark \ref{remark2.9} shows that the reflexivity is not necessary.

\begin{thm}\label{theorem2.8}
Let $E$ be a reflexive Banach space, then $f_k$ is s.w.l.s.c and Lipschitz for any positive number $k$ and any s.w.l.s.c proper function $f$ bounded below on $E$.
\end{thm}

\begin{proof}
It is known that $f_k$ is Lipschitz by Lemma \ref{lemma2.4}, we only need to verify that $f_k$ is s.w.l.s.c.
\par
Now, given any $x_0\in E$, let $\{x_n, n\in \mathbb{N}\}$ be any sequence such that it weakly converges to $x_0$. By the definition of $f_k$ there exists a sequence $\{y_n,n\in \mathbb{N}\}$ in $E$ such that
$$f_k(x_n)> f(y_n)+ k\|y_n-x_n\|- \frac{1}{n} $$ for each $n\geq 1$. {\hfill} (1)
\par
Since $f$ is bounded below, $f_k$ is Lipschitz, so that $\{f_k(x_n), n\in \mathbb{N}\}$ is bounded by the boundedness of $\{x_n, n\in \mathbb{N}\}$, then (1) forces $\{y_n, n\in \mathbb{N}\}$ to be bounded. By the Eberlin--Shmulyan theorem there exists a subsequence of $\{y_n,n \in \mathbb{N}\}$ such that the subsequence weakly converges to some $y_0\in E$, we can assume, without loss of generality, that the subsequence is just $\{y_n, n\in \mathbb{N}\}$, so that $\{y_n-x_n, n\in \mathbb{N}\}$ weakly converges to $y_0-x_0$. Since $f$ and $\|\cdot\|$ are both s.w.l.s.c, we have
\begin{align*}
&\underline{\lim}_n(f(y_n)+ k\|y_n-x_n\|) \\
&\geq f(y_0) + k\|y_0-x_0\| \\
&\geq f_k(x_0). \tag{2}
\end{align*}

\par
Combining (1) with (2) yields that $\underline{\lim}_n f_k(x_n) \geq f_k(x_0)$ .
\end{proof}

\begin{rem}\label{remark2.9}
Let us recall that a Banach space $E$ is a Schur space if any sequence $\{x_n, n\in \mathbb{N}\}$ in $E$ weakly converges to a point $x$ iff $\{x_n,n\in \mathbb{N}\}$ converges to $x$. For example, $l^1$ is a Schur space but not a reflexive Banach space. It is very easy to see that Theorem \ref{theorem2.8} still holds for any Schur space since the sequentially weakly lower semicontinuity and the lower semicontinuity coincide for any extended real--valued functions on a Schur space.
\end{rem}

\section{Proofs of the main results and some important remarks.}\label{section3}

\begin{proof}[Proof of Theorem \ref{theorem1.7}.]
$(1)\Rightarrow (2)$ is known as a special case of Proposition \ref{proposition1.5}.
\par
$(2) \Rightarrow (3)$. First, extend $f: A\rightarrow (-\infty, +\infty]$ to $E$ as follows:
$$ \tilde{f}(x)=\left\{
\begin{aligned}
&f(x), &\text{if}~ x\in A;  \\
&+\infty, & \text{if}~x\in E\setminus A.
\end{aligned}
\right.
$$
Then $\tilde{f}$ is still a proper l.s.c convex function and $\inf_{x\in E} \tilde{f}(x)=\inf_{x\in A}f(x)$. $\tilde{f}$ obviously satisfies (2), and hence there exists $x_0\in E$ such that $\tilde{f}(x_0)=\inf_{x\in E} \tilde{f}(x)= \inf_{x\in A}f(x)$, then $x_0$ must belong to $A$ since $\inf_{x\in A}f(x)< +\infty$, so that $f(x_0)= \tilde{f}(x_0)=\inf_{x\in A}f(x)$.
\par
$(3)\Rightarrow (4)$. Let $x_0\in A$ be any point. Further, let $A_1=\{x\in A~|~f(x)\leq f(x_0)\}$, then $A_1$ is a nonempty bounded closed convex subset of $A$ since $f$ is coercive. It is obvious that $f|_{A_1}$ satisfies (3) and $\inf_{x\in A_1} f|_{A_1}(x)=\inf_{x\in A} f(x)$, there exists $x_0\in A_1$ by (3) such that $f(x_0)=\inf_{x\in A_1} f|_{A_1}(x)=\inf_{x\in A} f(x)$.

\par
$(4)\Rightarrow (5)$ is clear.

\par
$(5)\Rightarrow (6)$ is clear by taking $A=E$ in (5) since a Lipschitz function must be bounded.

\par
$(6)\Rightarrow (7)$. Since $f$ is proper, there exists $x_0\in A$ such that $f(x_0)<+\infty$, further let $A_1=\{x\in A~|~ f(x)\leq f(x_0)\}$, then $A_1$ is a nonempty bounded closed convex subset of $A$ and $\inf_{x\in A_1} f(x)= \inf_{x\in A} f(x)$. Now, define $\hat{f}: E\rightarrow (-\infty, +\infty]$ as follows:
$$ \tilde{f}(x)=\left\{
\begin{aligned}
&f(x), &\text{if}~ x\in A_1;  \\
&+\infty, & \text{if}~x\in E\setminus A_1.
\end{aligned}
\right.
$$
Then $\hat{f}$ is still a proper l.s.c convex function and $\inf_{x\in E} \hat{f}(x)= \inf_{x\in A_1} f(x)$.
\par
Let $\{\hat{f}_n,n\in \mathbb{N}\}$ be the sequence of inf--convolutions of $\hat{f}$, then each $\hat{f}_n$ is a convex, coercive and Lipschitz function by (2), (3), (4) and (6) of Lemma \ref{lemma2.4} and Lemma \ref{lemma2.2}. Now, fix any $n\in \mathbb{N}$, then $\hat{f}_n$ satisfies (6) and thus there exists $x_0\in E$ such that $\hat{f}_n(x_0)= \inf_{x\in E} \hat{f}_n(x)$. Further, it immediately follows from Proposition \ref{proposition2.5} that $\hat{f}(x_0)=\inf_{x\in E} \hat{f}(x)$, in turn $x_0$ must belong to $A_1$ since $\inf_{x\in E} \hat{f}(x)\in \mathbb{R}$ (by Lemma \ref{lemma2.2}), then $f(x_0)=\inf _{x\in A_1} f(x)= \inf_{x\in A} f(x)$.

\par
$(7) \Rightarrow (8)$ is clear since $f$ satisfying (8) must be coercive.

\par
$(8) \Rightarrow (9)$. Define $f=-g+I_B$, then $f$ is a proper, l.s.c and convex function from $E$ to $(-\infty,+\infty]$. Further, choose $\beta \in \mathbb{R}$ for any given nondecreasing function $\varphi: [0, +\infty) \rightarrow [0,+\infty)$ satisfying $\lim_{t\rightarrow +\infty} \varphi(t)=+\infty$ such that $\varphi(1)+ \beta\leq -\|g\|$ (such functions $\varphi$ always exist, for example, take $\varphi(t)= \frac{1}{2}t$ for any $t\geq 0$), then $f(x)= -g(x) \geq -\|g\|\|x\| \geq -\|g\| \geq \varphi(1)+ \beta$ for any $x\in B$, and hence $f(x) \geq \varphi(\|x\|)+ \beta$ for any $x\in E$. By (8) there exists $x_0\in E$ such that $f(x_0)=\inf_{x\in E} f(x)= -\sup_{x\in B} g(x)$, then $x_0$ must belong to $B$, and hence $g(x_0)=\|g\|$.

\par
$(9) \Rightarrow (1)$ is obvious by the James theorem.
\end{proof}

\par
Theorem \ref{theorem1.7} only involves convex functions and their lower semicontinuity and it is well known that the lower semicontinuity of a convex function $f$ is lifted to the Lipschitz property of the Pasch--Hausdorff envelope $f_k$ of $f$, we can employ the regularization skill by inf--convolutions in the proof of $(6) \Rightarrow (7)$ of Theorem \ref{theorem1.7}. But Theorem \ref{theorem1.8} involves nonconvex functions and their sequentially weakly lower semicontinuity. Since we do not know if the regularization process by inf--convolutions can retain the sequentially weakly lower semicontinuity of a s.w.l.s.c nonconvex function defined on a general Banach space, we can not use the regularization skill by inf--convolutions to deduces (6) from (5) so that we are forced to prove Theorem \ref{theorem1.8} by completing the two circles: $(1)\Rightarrow (2)\Rightarrow (3)\Rightarrow (4)\Rightarrow (5)\Rightarrow (7)\Rightarrow (8)\Rightarrow (9)\Rightarrow (1)$ and $(1)\Rightarrow (2)\Rightarrow (3)\Rightarrow (4)\Rightarrow (6)\Rightarrow (7)\Rightarrow (8)\Rightarrow (9)\Rightarrow (1)$, we eventually complete the proof of $(5)\Leftrightarrow (6)$. Fortunately, we can use the regularization skill by inf--convolutions in the proof of $(8) \Rightarrow (9)$ of Theorem \ref{theorem1.8} since $f=: -g+ I_B$ is a l.s.c convex function for any $g\in E^*$ and each $f_n$ is convex so that $f_n$ is also s.w.l.s.c.

\begin{proof}[Proof of Theorem \ref{theorem1.8}.]
 $(1) \Rightarrow (2)$ is just Proposition \ref{proposition1.5}.
\par
$(2)\Rightarrow (3)\Rightarrow (4)$ are completely similar to the proofs of $(2)\Rightarrow (3)\Rightarrow (4)$ of Theorem \ref{theorem1.7}, so are omitted.

\par
Both $(4)\Rightarrow (5)$ and $(4)\Rightarrow (6)$ are obvious.

\par
$(5)\Rightarrow (7)$ is also clear.

\par
$(6) \Rightarrow (7)$. Let $A$ and $f$ be assumed as in (7), further arbitrarily choose $x_0\in A$ and let $A_1=\{x\in A~|~f(x)\leq f(x_0)\}$, then $A_1$ is a nonempty bounded and sequentially weakly closed subset of $A$ by the sequentially weakly lower semicontinuity and coercivity of $f$. Now, let $m=\inf\{f(x)~|~x\in A_1\}$, then $m\in R$ since $f$ is bounded. Further, choose $\beta\in \mathbb{R}$ for any given $\varphi$ satisfying (6) such that $\varphi(h)+ \beta \leq m$, where $h=\sup\{\|x\|~|~x\in A_1\}$, so that $f(x)\geq m \geq \varphi(h)+ \beta\geq \varphi(\|x\|)+ \beta$ for any $x\in A_1$.

\par
Let us define the extension $\tilde{f}$ of $f$ as follows:
$$ \tilde{f}(x)=\left\{
\begin{aligned}
&f(x), &\text{if}~ x\in A_1;  \\
&+\infty, & \text{if}~x\in E\setminus A_1.
\end{aligned}
\right.
$$
Then $\tilde{f}$ is s.w.l.s.c and $\tilde{f}(x)\geq \varphi(\|x\|)+ \beta$ for any $x\in E$, by (6) there exists $x_0\in E$ such that $\tilde{f}(x_0)=\inf _{x\in E} \tilde{f}(x)= \inf_{x\in A_1} f(x)=m$, which means that $x_0\in A_1$ and hence $f(x_0)=\inf_{x\in A_1} f(x)$. Besides, it is obvious that $\inf_{x\in A_1} f(x) =\inf_{x\in A} f(x)$, so $f(x_0)= \inf_{x\in A} f(x)$.

\par
$(7) \Rightarrow (8)$ is obvious since a Lipschitz function is bounded.

\par
$(8) \Rightarrow (9)$. Let $g\in E^*$ and define $f: E\rightarrow (-\infty,+\infty]$ by $f=-g+I_B$, where $B=\{x\in E~|~\|x\|\leq 1\}$, then $f$ is proper convex and l.s.c. Further, let $\{f_n, n\in \mathbb{N}\}$ be the sequence of inf--convolutions of $f$, then each $f_n$ is convex, coercive and Lipschitz since $f$ is convex and bounded below and since $f$ also satisfies (6) of Lemma \ref{lemma2.4}, so that $f_n$ is, of course, s.w.l.s.c. Now, fix any $n$ in $\mathbb{N}$, then there exists $x_0\in E$ such that $f_n(x_0)= \inf_{x\in E} f_n(x)$ by (8). Thus we have $f(x_0)= \inf_{x\in E} f(x)$ again by Proposition \ref{proposition2.5}, which forces $x_0\in B$ and hence $g(x_0)=\|g\|.$

\par
$(9) \Rightarrow (1)$ is clear again by the James theorem.
\end{proof}

\par
Now, we come to the proof of Theorem \ref{theorem1.9}, whose key part is the proof that (10) implies (1).

\begin{proof}[Proof of Theorem \ref{theorem1.9}.]
Similarly to the proof of Theorem \ref{theorem1.8}, one can see $(1)\Rightarrow (2)\Rightarrow (3)\Rightarrow (4)\Rightarrow (5)\Leftrightarrow (6)\Rightarrow (7)\Rightarrow (8)$. Now, we prove that $(8)\Rightarrow (10)\Rightarrow (1)$ and that $(1)\Rightarrow (9)\Rightarrow (10)\Rightarrow (1)$ as follows.

\par
$(8)\Rightarrow (10)$. Let $A$ be a nonempty bounded closed subset of $E$ and $f: A\rightarrow \mathbb{R}$ be continuous and bounded below, define $\tilde{f}: E\rightarrow (-\infty,+\infty]$ as follows:
$$ \tilde{f}(x)=\left\{
\begin{aligned}
&f(x), &\text{if}~ x\in A;  \\
&+\infty, & \text{if}~x\in E\setminus A.
\end{aligned}
\right.
$$
Then $\tilde{f}$ is bounded below and l.s.c.

\par
Consider the sequence $\{\tilde{f}_n, n\in \mathbb{N}\}$ of inf--convolutions of $\tilde{f}$, then each $\tilde{f}_n$ is coercive and Lipschitz by (4) and (6) of Lemma \ref{lemma2.4}. For a fixed $n\in \mathbb{N}$, by (8) there exists $x_0\in E$ such that $\tilde{f}_n(x_0)=\inf_{x\in E} \tilde{f}_n(x)$, further by Proposition \ref{proposition2.5} one has $\tilde{f}(x_0)= \inf_{x\in E} \tilde{f}(x)$, which in turn implies $x_0\in A$ and $f(x_0)=\inf_{x\in A} f(x)$.

\par
$(10) \Rightarrow (1)$. Let $A$ be a nonempty bounded closed subset of $E$ and $\{x_n, n\in \mathbb{N}\}$ be any sequence in $A$. For any given $k\in \mathbb{N}$, define $f_k: A \rightarrow [0,+\infty)$ as follows:
$$f_k(x)=\inf\{\|x_n-x\|: n\geq k\} ~\text{for any }x\in A.$$
\par
Further, let $D(A)$ denote the diameter of $A$ and define $f: A\rightarrow [0,+\infty)$ as follows:
$$f(x)=\sum_{k=1}^{\infty} \frac{f_k(x)}{2^k D(A)}~\text{for any }x\in A.$$
\par
Since $f_k(x)=d(x,X_k)$, where $X_k=\{x_n, n\geq k\}$, each $f_k$ is continuous, which also implies that $f$ is continuous since the series defining $f$ converges uniformly on $A$. Next, we prove $\inf_{x\in A} f(x)=0$ as follows.
\par
For any given positive number $\varepsilon$, choose a sufficiently large $k_0$ such that $\sum_{k=k_0}^{\infty}\frac{1}{2^k}< \varepsilon$, then it is very easy to see that $f(x_{k_0})< \varepsilon$ by observing $f_k(x_{k_0})=0$ for each $k$ with $1\leq k\leq k_0-1$ and $\frac{f_k(x)}{2^k D(A)} \leq \frac{1}{2^k}$ for each $k\in \mathbb{N}$ and $x\in A$. Thus $\inf_{x\in A} f(x)< \varepsilon$ for any $\varepsilon>0$, namely $\inf_{x\in A} f(x)=0$.
\par
Now, by (10) there exists $x_0\in A$ such that $f(x_0)=0$, which must imply that $f_k(x_0)=0$ for each $k\in \mathbb{N}$. Then, by the definition of $f_k$ there must exist a subsequence $\{x_{n_k}, k\in \mathbb{N}\}$ of $\{x_n, n\in \mathbb{N}\}$ such that $\lim_{k\rightarrow \infty} \|x_{n_k}-x_0\|=0$, so that $A$ is sequentially compact. Finally, by the Riesz lemma $E$ must be finite--dimensional.
\par
$(1)\Rightarrow (9)$ is the classical Weierstrass theorem.

\par
$(9)\Rightarrow (10)$ is clear.

\par
$(10)\Rightarrow (1)$ has been proved as above.
\end{proof}

\par
Finally, we prove Corollary \ref{corollary1.10}.

\begin{proof}[Proof of Corollary \ref{corollary1.10}.]
Let $E$ be an infinite--dimensional reflexive Banach space. If each coercive Lipschitz function $f: E\rightarrow \mathbb{R}$ is s.w.l.s.c, then $f$ attains its minimum by Theorem \ref{theorem1.8}, which implies that $E$ must be finite--dimensional by Theorem \ref{theorem1.9}, so that a contradiction will be produced.
\end{proof}

\par
Finally, let us end this paper with the following important remarks.

\begin{rem}\label{remark3.1}
Since a normed or Banach space over the field $\mathbb{C}$ of complex numbers can be also regarded as one over $\mathbb{R}$, all the results of this paper still hold for the corresponding complex spaces, as pointed out by Br\'{e}zis in \cite[Proposition 11.22, p.361]{Bre}, only by replacing $g$ with $Re(g)$ when a bounded linear functional $g$ is involved, where $Re(g)$ stands for the usual real part of $g$.
\end{rem}

\begin{rem}\label{remark3.2}
If $\varphi$ in (8) of Theorem \ref{theorem1.7}, (6) of Theorem \ref{theorem1.8} and (6) of Theorem \ref{theorem1.9} is fixed, respectively, then the three theorems still hold, since fixing $\varphi$ does not have any effect on our proofs of these theorems although the corresponding class of functions will be restricted. For example, we can change (8) of Theorem \ref{theorem1.7} as follows: let $\varphi$ be a nondecreasing function from $[0,+\infty)$ to $[0,+\infty)$ satisfying $\lim_{t\rightarrow +\infty} \varphi(t)= +\infty$, then for each proper l.s.c convex function $f: E \rightarrow (-\infty,+\infty]$ satisfying $f(x)\geq \varphi(\|x\|)+ \beta$ for any $x\in E$ and for some $\beta\in \mathbb{R}$ there exists $x_0\in E$ such that $f(x_0)=\inf_{x\in E} f(x)$.
\end{rem}

\begin{rem}\label{remark3.3}
The converse of the classical Weierstrass theorem, namely (9) $\Rightarrow$ (1) of Theorem \ref{theorem1.9}, only holds within the scope of normed spaces by means of the famous Riesz lemma. But the converse does not hold for metrizable linear topological spaces, for example, let $\Omega$ be a nonempty open subset of the complex plane and $H(\Omega)$ the complete metrizable locally convex space of holomorphic functions in $\Omega$, endowed with the topology of uniform convergence on compact subsets of $\Omega$. It is well known from \cite[p.34]{Rudin} that each bounded closed subset of $H(\Omega)$ is compact, and hence (9) of Theorem \ref{theorem1.9} holds but $H(\Omega)$ is infinite dimensional, which is because $H(\Omega)$ is not normable!
\end{rem}
%
%
%
%
%
%

\bibliographystyle{amsplain}

\end{document}